\documentclass[12pt]{amsart}\hoffset=-3cm\leftmargin-1cm
\usepackage{amsmath, amssymb}
\usepackage{hyperref}
\usepackage[displaymath, mathlines]{lineno}
\newtheorem{theorem}{Theorem}[section]
\newtheorem{corollary}[theorem]{Corollary}
\newtheorem{lemma}[theorem]{Lemma}

\numberwithin{equation}{section}

\begin{document}
\title [On starlikeness of $p$-valent  analytic functions ]
       {On starlikeness of $p$-valent  analytic functions}

\author[M. Nunokawa, J. Sok\'{o}\l]
       {Mamoru Nunokawa$^1$ and Janusz Sok\'{o}\l$^{*,2}$}

\address{$^1$  University of Gunma, Hoshikuki-cho 798-8,
         Chuou-Ward, Chiba, 260-0808, Japan}
         \email{mamorununokawa1983@gmail.com}

\address{$^2$ University of Rzesz\'{o}w, College Natural
         Sciences,
         ul. Prof. Pigonia 1, 35-310 Rzesz\'{o}w, Poland}
         \email{jsokol@ur.edu.pl}

\keywords{ analytic functions, Ozaki's condition, univalent, multivalent, starlike functions\\
            *corresponding author jsokol@ur.edu.pl}

\dedicatory{Dedicated to  Sachiko Nunokawa}

\begin{abstract}
The  known Ozaki's condition  says that
$\mathfrak{Re}\left\{f^{(p)}(z)\right\}>0$ for $|z|<1$
implies that $f(z)=z^p+a_{p+1}z^{p+1}+\cdots$ is at most $p$-valent in $\mathbb D$.
In this paper prove an extension of Ozaki's condition.
Also, we shall determine the new sufficient conditions for
functions to be in the class of $p$-valent starike  of order $\alpha$.
~~
\end{abstract}
\maketitle

\section{Introduction}
Let ${\mathcal H}$ denote the class of analytic functions in the
unit disc ${{\mathbb D}}=\{z:\ |z|<1\}$ on the complex plane
${\mathbb{C}}$. We will use the following notations:
\begin{equation*}\label{J}
\left\{
\begin{array}{rcl}
   J_{CV}(f;z)&:=&1+\frac{zf''(z)}{f'(z)},\\
   J_{ST}(f;z)&:=&\frac{zf'(z)}{f(z)}.
\end{array}
\right.
\end{equation*}
A function $f$ analytic in a domain $D \in\mathbb{C}$
is called $p$-valent in $D$, if for every complex number $w$, the equation $%
f(z)= w$ has at most $p$ roots in $D$, so that there exists a complex number $%
w_0$ such that the equation $f(z) = w_0$ has exactly $p$ roots in
$D$. We denote by $\mathcal{H}$  the class of functions $f(z)$ which
are holomorphic in the open unit unit $\mathbb D=\{z\in \mathbb C:
|z|<1\}$. Denote by $\mathcal A_p$, $p\in\mathbb N=\{1,2,\ldots\} $,
the class of functions $f(z)\in\mathcal H$ given by
\begin{equation*}\label{fp}
    f(z)=z^p+ \sum_{n=p+1}^{\infty}a_{n}z^{n},\quad (z\in\mathbb D).
\end{equation*}
Let $\mathcal A=\mathcal A_1$. Let $\mathcal{S}$ denote the class
of all functions in $\mathcal{A}$
which are univalent in $\mathbb D$. Also let $\mathcal{S}_p^{\ast }(\alpha )$ and $\mathcal{C}_p%
(\alpha )$ be the subclasses of $\mathcal{A}_p$ consisting of all
$p$-valent functions which are strongly starlike and strongly convex
of order $\alpha $, $0< \alpha \leq1$, defined as
\begin{eqnarray*}
    \mathcal{S}_p^{\ast }(\alpha )
    &=&\left\{ f(z)\in \mathcal{A}(p):\left|\arg \left\{ J_{ST}(f;z)\right\}\right| <\frac{\alpha\pi}{2} ,\ z\in \mathbb D\right\}, \\
    && \\
    \mathcal{C}_p(\alpha ) &=&\left\{ f(z)\in \mathcal{A}(p):zf'(z)/p\in
    \mathcal{S}_p^{\ast }(\alpha )\right\} .
\end{eqnarray*}%
Note that $\mathcal{S}_1^{\ast }(1)=\mathcal{S}^{\ast }$ and $\mathcal{C}_1(1)=%
\mathcal{C}$, where $\mathcal{S}^{\ast }$ and $\mathcal{C}$ are
usual classes of starlike and convex functions respectively. The
known Ozaki's condition says that
\begin{equation*}
 \mathfrak{Re}\left\{f^{(p)}(z)\right\}>0,\quad (z\in
 \mathbb D)
\end{equation*}
follows that $f(z)$ is at most $p$-valent in $\mathbb D$.

In this paper we shall determine the new sufficient conditions for
functions to be in the class $\mathcal{S}_p^{\ast }(\alpha )$. The
key in proving is  Nunokawa's lemma \cite{Nunokawa92} and the
following Lemma \ref{l1} which generalizes it, see also
\cite{FuSa80}, \cite{Nunokawa93}. Note, that the geometric
interpretation of the Nunokawa's is similar to the geometric
interpretation of the Jack's lemma, \cite{Jack}.

\begin{lemma}\label{l1}\cite{Nunokawa93}
Let   $q(z)=1+\sum_{n=m}^{\infty}c_nz^n$, $c_m\neq 0$ be an analytic
function in $\mathbb{D}$ with  $q(z)\neq 0$. If there exists a point
$z_0$, $|z_0|<1$, such that $|{\arg}
\left\{q(z)\right\}|<\pi\gamma/2$ for $|z|<|z_0|$ and
$|{\arg}\left\{ q(z_0)\right\}|=\pi\gamma/2$ for some $\gamma>0$,
then we have
\begin{equation}\label{1l11}
     \frac{z_0q'(z_0)}{q(z_0)}=\frac{2ik\arg\left\{q(z_0)\right\}}{\pi},
\end{equation}
for some $k\geq m(a+a^{-1})/2\geq(a+a^{-1})/2\geq1$, where $
\left\{q(z_0)\right\}^{1/\gamma}=\pm ia$, and $a>0$.
\end{lemma}
 Lemma \ref{l1} was generalized in \cite{JSLTS} by considering a hypothesis that ${\arg}
\left\{q(z)\right\}\in (\pi\gamma_1/2, \pi\gamma_2/2)$ instead of
$|{\arg} \left\{q(z)\right\}|<\pi\gamma/2$.

In this paper we need also the following lemmas.

\begin{lemma}\label{l2}\cite[Th.5]{Nunokawa87Tsukuba} If $f(z)\in\mathcal A_p$,
then for all $z\in\mathbb D$, we have
\begin{equation}\label{1l1}
    \mathfrak{Re}\left\{\frac{zf^{(p)}(z)}{f^{(p-1)}(z)}\right\}>0\quad
    \Rightarrow
    \quad \forall
    k\in\{1,\ldots,p\}:\quad\mathfrak{Re}\left\{\frac{zf^{(k)}(z)}{f^{(k-1)}(z)}\right\}>0.
\end{equation}
\end{lemma}
\begin{lemma}\label{l3}\cite[Th.1]{Nunokawa87Tsukuba} If $f(z)\in\mathcal A_p$,
then for all $z\in\mathbb D$, we have
\begin{equation}\label{1l2}
    \mathfrak{Re}\left\{p+\frac{zf^{(p+1)}(z)}{f^{(p)}(z)}\right\}>0\quad (z\in\mathbb
    D),
\end{equation}
then f(z) is $p$-valent in $\mathbb D$ and
\begin{equation*}
     \forall
    k\in\{1,\ldots,p-1\}:\quad\mathfrak{Re}\left\{k+\frac{zf^{(k+1)}(z)}{f^{(k)}(z)}\right\}>0
    \quad (z\in\mathbb
    D).
    .
\end{equation*}

\end{lemma}

\section{Main results}


\begin{theorem}\label{t1}
Let $f(z)=z^p+\sum_{n=p+1}^\infty a_nz^n$ be analytic in $\mathbb
D$.  If
\begin{equation}\label{1t1}
    \left|\arg\{f^{(p)}(z)\}\right|<
   \frac{\pi}{2}\left(\alpha_1+\frac{2}{\pi}\tan^{-1}\alpha_1\right)
    , \quad z\in\mathbb D
\end{equation}
for some $\alpha_1\in(0,1]$, then
\begin{equation}\label{2t1}
     \left|\arg\left\{\frac{f^{(p-1)}(z)}{z}\right\}\right|<
     \frac{\alpha_1\pi}{2}
    , \quad z\in\mathbb D.
\end{equation}

\end{theorem}
\begin{proof}
Write
\begin{equation*}
    p(z)=\frac{f^{(p-1)}(z)}{p!z}, \quad z\in\mathbb D.
\end{equation*}
Then $p(0)=1$ and
\begin{equation*}
    f^{(p)}(z)=p!(p(z)+zp'(z)).
\end{equation*}
From \eqref{1t1}, we have
\begin{equation}\label{3t1}
    \left|\arg\{p(z)+zp'(z)\}\right|<
   \frac{\pi}{2}\left(\alpha_1+\frac{2}{\pi}\tan^{-1}\alpha_1\right)
    , \quad z\in\mathbb D.
\end{equation}
 If there exists a point
$z_0\in\mathbb D$, such that
\begin{eqnarray}\label{33t1}
     |\arg\{p(z)\}|&<&\frac{\alpha_1\pi}{2},\quad |z|<|z_0|,\nonumber\\
    |\arg\{p(z_0)\}|&=&\frac{\alpha_1\pi}{2},
\end{eqnarray}
then from Lemma \ref{l1}, we have
\begin{equation*}
    \frac{z_0p'(z_0)}{p(z_0)}=\frac{2ik\arg\left\{p(z_0)\right\}}{\pi}
\end{equation*}
for some $k\geq 1$. For the case
$\arg\{p(z_0)\}=\frac{\alpha_1\pi}{2}$, we have
\begin{equation*}
    \arg\left\{1+\frac{z_0p'(z_0)}{p(z_0)}\right\}=\arg\left\{1+i\alpha_1
    k\right\}\in[\tan^{-1}\alpha_1,\pi/2)
\end{equation*}
because $k\geq 1$. Therefore, we have
\begin{equation}\label{xxx}
     \arg\left\{
    p(z_0)\left(1+\frac{z_0p'(z_0)}{p(z_0)}\right)\right\}\\
    =\arg\left\{
    p(z_0)\right\}+ \arg\left\{1+\frac{z_0p'(z_0)}{p(z_0)}\right\}.
\end{equation}

This gives
\begin{eqnarray*}
  \arg\left\{
    p(z_0)+z_0p'(z_0)\right\}
    &=& \arg\left\{
    p(z_0)\left(1+\frac{z_0p'(z_0)}{p(z_0)}\right)\right\}\\
    &=&\arg\left\{
    p(z_0)\right\}+ \arg\left\{1+\frac{z_0p'(z_0)}{p(z_0)}\right\}\\
    &=&\frac{\alpha_1\pi}{2}+\arg\left\{1+i\alpha_1 k\right\} \\
   &\geq&\frac{\alpha_1\pi}{2}+\tan^{-1}\alpha_1\\
   &=&\frac{\pi}{2}\left(\alpha_1+\frac{2}{\pi}\tan^{-1}\alpha_1\right).
\end{eqnarray*}
This contradicts \eqref{3t1}. For the case
$\arg\{p(z_0)\}=-\frac{\alpha_1\pi}{2}$, applying the same method as
the above we can obtain
\begin{equation*}
    \arg\left\{
    p(z_0)+z_0p'(z_0)\right\}\leq
    -\frac{\pi}{2}\left(\alpha_1+\frac{2}{\pi}\tan^{-1}\alpha_1\right)
\end{equation*}
This contradicts \eqref{3t1} too. Therefore, we have
\begin{equation*}
     \left|\arg\left\{p(z)\right\}\right|<
     \frac{\alpha_1\pi}{2}
    , \quad z\in\mathbb D
\end{equation*}
which ends the proof.
\end{proof}

Theorem \ref{t1} with $\alpha_1=1$ gives the following corollary.
\begin{corollary}\label{c1}  Let $f(z)=z^p+\sum_{n=p+1}^\infty a_nz^n$ be analytic in $\mathbb
D$.  If
\begin{equation}\label{1c1}
    \left|\arg\{f^{(p)}(z)\}\right|<
   \frac{3\pi}{4}
    , \quad z\in\mathbb D,
\end{equation}
then
\begin{equation}\label{2c1}
     \left|\arg\left\{\frac{f^{(p-1)}(z)}{z}\right\}\right|<
     \frac{\pi}{2}
    , \quad z\in\mathbb D
\end{equation}
and
\begin{equation}\label{3c1}
     \mathfrak{Re}\left\{\frac{f^{(p-k-1)}(z)}{z^{k+1}}\right\}>0
    , \quad z\in\mathbb D
\end{equation}
for all $k\in\{0,\ldots,p-1\}$ and if $p\geq 2$, then $f(z)$ is $p$-valent in $\mathbb
D$.
\end{corollary}

\begin{proof}
Theorem \ref{t1} with $\alpha_1=1$ and \eqref{1c1} imply
\eqref{2c1}. From \cite[Lemma 5]{Nunokawa87Tsukuba} condition
\eqref{2c1} implies \eqref{3c1} for all $k\in\{0,\ldots,p-1\}$. If we denote $s=p-k-1$, then $s\in\{0,\ldots,p-1\}$ and
\eqref{3c1} becomes
\begin{equation*}
    \mathfrak{Re}\left\{\frac{zf^{(s)}(z)}{z^{p-s+1}}\right\}
          >0, \quad z\in\mathbb D,
\end{equation*}
where $z^{p-s+1}$ is a $(p-s+1)-$ valent starlike function. Therefore, by \cite[Theorem 8]{Nunokawa87Tsukuba}, $f(z)$ is $p$-valent in
$\mathbb D$, whenever $s\geq 1$ which is possible for $p\geq 2$ only.
\end{proof}

 The  known Ozaki's condition \cite{Ozaki35} says that
\begin{equation*}
 \mathfrak{Re}\left\{f^{(p)}(z)\right\}>0,\quad (z\in
 \mathbb D)
\end{equation*}
follows that $f(z)\in\mathcal A-p$ is at most $p$-valent in $\mathbb D$. This becomes Noshiro-Warshawski result for $p=1$, \cite{Nosh, Warshawski}. We see that Corollary \ref{c1} is
an improvement of this result in the case $p\geq 2$. Some related resuls where published in \cite{Nunokawa101, Nunokawa136, Nunokawa178}.


\begin{corollary}\label{c2}   Let $f(z)=z^p+\sum_{n=p+1}^\infty a_nz^n$ be analytic in $\mathbb
D$. If
\begin{equation}\label{1c2}
    \left|\arg\{f^{(p)}(z)\}\right|<
   \frac{\pi}{2}\left(\gamma_0+\frac{2}{\pi}\tan^{-1}\gamma_0\right)=\frac{\pi \cdot0.6\ldots}{2}
    , \quad z\in\mathbb D,
\end{equation}
where $\gamma_0=0.383\ldots$ is the root of the equation
\begin{equation*}
    2\gamma+\frac{2}{\pi}\tan^{-1}\gamma=1,
\end{equation*}
then $f(z)$ is $p$-valently starlike in $\mathbb D$ in other words
$f(z)\in\mathcal{S}_p^{\ast }(0)$.
\end{corollary}
\begin{proof}
From \eqref{1c2} and Theorem \ref{t1}, we have
\begin{equation*}
     \left|\arg\left\{\frac{f^{(p-1)}(z)}{z}\right\}\right|<
     \frac{\gamma_0\pi}{2}
    , \quad z\in\mathbb D.
\end{equation*}
Then, we have
\begin{eqnarray*}
  \left|\arg\frac{zf^{(p)}(z)}{f^{(p-1)}(z)}\right|
  &=& \left|\arg f^{(p)}(z)+ \arg\frac{z}{f^{(p-1)}(z)}\right| \\
  &\leq& \left|\arg f^{(p)}(z)\right|+\left|\arg\frac{f^{(p-1)}(z)}{z}\right| \\
  &<&\frac{\pi}{2}\left(2\gamma_0+\frac{2}{\pi}\tan^{-1}\gamma_0\right)\\
  &=&\frac{\pi}{2},  \quad z\in\mathbb D.
\end{eqnarray*}
From \ref{l2}, this implies that
\begin{equation*}
    \left|\arg\frac{zf'(z)}{f(z)}\right|<
    \frac{\pi}{2},  \quad z\in\mathbb D
\end{equation*}
so $f(z)$ is $p$-valently starlike in $\mathbb D$,
$f(z)\in\mathcal{S}_p^{\ast}(0)$.
\end{proof}


Consider the  number sequence $\{\alpha_n\}$ such that
\begin{equation*}
    \alpha_k+\frac{2}{\pi}\tan^{-1}\frac{\alpha_k}{k}=\alpha_{k-1},\quad
    \alpha_0=3/2.
\end{equation*}
We want to prove that
\begin{equation}\label{limit}
    \lim_{n\rightarrow\infty}\alpha_n=0
\end{equation}
To show that this limit is $0$ consider a sequence $\{x_n\}$ such
that
\begin{equation*}
    x_k+\frac{1}{\pi}\frac{x_k}{k}=x_{k-1},\quad
    x_0=2.
\end{equation*}
Then it is easy to see that
\begin{eqnarray*}
    \alpha_n&<&x_n\\
        &=&\frac{2}{\left(1+\frac{1}{\pi}\right)\left(1+\frac{1}{2\pi}\right)
    \left(1+\frac{1}{3\pi}\right)\ldots\left(1+\frac{1}{n\pi}\right)}.
\end{eqnarray*}
Applying the inequality $\log x > (x-1)/x $ to the denominator, we
obtain
\begin{eqnarray*}
    &~&\log\left[\left(1+\frac{1}{\pi}\right)\left(1+\frac{1}{2\pi}\right)
    \left(1+\frac{1}{3\pi}\right)\ldots\left(1+\frac{1}{n\pi}\right)\right]\\
    &=&\log\left(1+\frac{1}{\pi}\right)
    +\log\left(1+\frac{1}{2\pi}\right)
    +\log\left(1+\frac{1}{3\pi}\right)+
    \ldots+\log\left(1+\frac{1}{n\pi}\right)\\
    &>&\frac{1}{\pi}\sum_{k=1}^n\frac{1}{k+1/\pi}\rightarrow\infty
\end{eqnarray*}
as $n\rightarrow\infty $ because this is the partial sum of a
harmonic series, hence the limit of $\{x_n\}$ is $0$ and the limit
of $\{\alpha_n\}$ is zero too. The sequence  sequence $\{\alpha_n\}$
occurs in Theorem \ref{t3}.

\begin{theorem}\label{t3}
Let $f(z)=z^p+\sum_{n=p+1}^\infty a_nz^n$ be analytic in $\mathbb
D$. If
\begin{equation}\label{1t3}
    \left|\arg\{f^{(p)}(z)\}\right|<\frac{\pi\alpha_0}{2}
    , \quad z\in\mathbb D.
\end{equation}
for some $\alpha_0\in(0,3/2]$ then, we have
\begin{equation}\label{111t3}
    \left|\arg\left\{\frac{f^{(p-k)}(z)}{z^k}\right\}\right|
    <\frac{\pi\alpha_k}{2}, \quad z\in\mathbb D,
\end{equation}
where
\begin{equation}\label{11t3}
    \alpha_0\in(0,3/2],\quad
    \alpha_k+\frac{2}{\pi}\tan^{-1}\frac{\alpha_k}{k}=\alpha_{k-1}
\end{equation}
for all  $k\in\{1,\ldots,p\}$.
\end{theorem}

\begin{proof}
It is easy to see that the case $k=1$ was considered in Theorem
\ref{t1} with
\begin{equation*}
    \alpha_{0}=\alpha_1+\frac{2}{\pi}\tan^{-1}\alpha_1.
\end{equation*}
 Hence, we have
\begin{equation}\label{2t3}
    \left|\arg\left\{\frac{f^{(p-1)}(z)}{z}\right\}\right|
    <\frac{\pi\alpha_1}{2}, \quad z\in\mathbb D,
\end{equation}

 Next, let us put
\begin{equation*}
    q_2(z)=\frac{2f^{(p-2)}(z)}{p!z^2}, \quad z\in\mathbb D.
\end{equation*}
Then $q_2(0)=1$ and
\begin{equation*}
    \frac{f^{(p-1)}(z)}{z}=\frac{p!}{2}q_2(z)\left(2+\frac{zq_2'(z)}{q_2(z)}\right).
\end{equation*}
and so
\begin{equation}\label{3t3}
    \arg\left\{\frac{f^{(p-1)}(z)}{z}\right\}
    =\arg\left\{q_2(z)\left(2+\frac{zq_2'(z)}{q_2(z)}\right)\right\}
\end{equation}
We apply the same method as in the case $k=1$. If there exists a
point $z_2\in\mathbb D$, such that
\begin{eqnarray*}
     |\arg\{q_2(z)\}|&<&\frac{\alpha_2\pi}{2},\quad |z|<|z_2|,\nonumber\\
    |\arg\{q_2(z_2)\}|&=&\frac{\alpha_2\pi}{2},
\end{eqnarray*}
then from Lemma \ref{l1}, we have
\begin{equation*}
    \frac{z_2q_2'(z_2)}{q_2(z_2)}=\frac{2ik\arg\left\{q_2(z_2)\right\}}{\pi}
\end{equation*}
for some $k\geq 1$. For the case
$\arg\{q_2(z_2)\}=\frac{\alpha_2\pi}{2}$, we have from\eqref{3t3}
\begin{eqnarray*}
  \arg\left\{\frac{f^{(p-1)}(z_2)}{z_2}\right\}&=&\arg\{q_2(z_2)\}
  +\arg\left\{2+\frac{z_2q_2'(z_2)}{q_2(z_2)}\right\}\\
    &=& \frac{\alpha_2\pi}{2}+\arg\left\{2+i\alpha_2 k\right\} \\
   &\geq&\frac{\alpha_2\pi}{2}+\tan^{-1}\frac{\alpha_2}{2}\\
   &=&\frac{\pi}{2}\left(\alpha_2+\frac{2}{\pi}\tan^{-1}\frac{\alpha_2}{2}\right)\\
   &=&\frac{\pi\alpha_1}{2}.
\end{eqnarray*}
This contradicts \eqref{2t3}. For the case
$\arg\{q_2(z_2)\}=-\frac{\alpha_2\pi}{2}$, applying the same method
as the above we can obtain
\begin{equation*}
    \arg\left\{\frac{f^{(p-1)}(z_2)}{z_2}\right\}\leq
    -\frac{\pi}{2}\left(\alpha_2+\frac{2}{\pi}\tan^{-1}\frac{\alpha_2}{2}\right)
\end{equation*}
This contradicts \eqref{2t3} too. Therefore, we have
\begin{equation}\label{4t3}
     \left|\arg\left\{q_2(z)\right\}\right|
     =\left|\arg\frac{2f^{(p-2)}(z)}{p!z^2}\right|<
     \frac{\alpha_2\pi}{2}
    , \quad z\in\mathbb D,
\end{equation}
which ends the proof for $k=2$.

For the case $k=3$, we put
\begin{equation*}
    q_3(z)=\frac{2\cdot3f^{(p-3)}(z)}{p!z^3}, \quad z\in\mathbb D.
\end{equation*}
Then $q_3(0)=1$ and
\begin{equation*}
    \frac{f^{(p-2)}(z)}{z^2}=\frac{p!}{6}q_3(z)\left(3+\frac{zq_3'(z)}{q_3(z)}\right).
\end{equation*}
and so
\begin{equation*}
    \arg\left\{\frac{f^{(p-2)}(z)}{z^2}\right\}
    =\arg\left\{q_3(z)\left(3+\frac{zq_3'(z)}{q_3(z)}\right)\right\}
\end{equation*}
We apply the same method as in  cases $k=1,2$. If there exists a
point $z_3\in\mathbb D$, such that
\begin{eqnarray*}
     |\arg\{q_3(z)\}|&<&\frac{\alpha_3\pi}{2},\quad |z|<|z_3|,\nonumber\\
    |\arg\{q_3(z_3)\}|&=&\frac{\alpha_3\pi}{2},
\end{eqnarray*}
then from Lemma \ref{l1}, we have
\begin{equation*}
    \frac{z_3q_3'(z_3)}{q_3(z_3)}=\frac{2ik\arg\left\{q_3(z_3)\right\}}{\pi}
\end{equation*}
for some $k\geq 1$. For the case
$\arg\{q_3(z_3)\}=\frac{\alpha_3\pi}{2}$, we have
\begin{eqnarray*}
  \arg\left\{\frac{f^{(p-2)}(z_3)}{z^2_3}\right\}&=&\arg\{q_3(z_3)\}
  +\arg\left\{3+\frac{z_3q_3'(z_3)}{q_3(z_3)}\right\}\\
    &=& \frac{\alpha_3\pi}{2}+\arg\left\{3+i\alpha_3 k\right\} \\
   &\geq&\frac{\alpha_3\pi}{2}+\tan^{-1}\frac{\alpha_3}{3}\\
   &=&\frac{\pi}{2}\left(\alpha_3+\frac{2}{\pi}\tan^{-1}\frac{\alpha_3}{3}\right)\\
   &=&\frac{\pi\alpha_2}{2}.
\end{eqnarray*}
This contradicts \eqref{4t3}. For the case
$\arg\{q_3(z_3)\}=-\frac{\alpha_3\pi}{2}$, applying the same method
as the above we can obtain a contradiction. Therefore, we have
\begin{equation}\label{5t3}
     \left|\arg\left\{q_3(z)\right\}\right|
     =\left|\arg\frac{2\cdot3f^{(p-3)}(z)}{p!z^3}\right|
     <\frac{\alpha_3\pi}{2}
    , \quad z\in\mathbb D.
\end{equation}
This ends the proof for the case $k=3$. Applying the same method as
the above again and again we can obtain
\begin{equation*}
     \left|\arg\left\{\frac{f^{(p-k)}(z)}{z^k}\right\}\right|<
     \frac{\alpha_k\pi}{2}
    , \quad z\in\mathbb D
\end{equation*}
and
\begin{equation*}
    \alpha_k+\frac{2}{\pi}\tan^{-1}\frac{\alpha_k}{k}=\alpha_{k-1}
\end{equation*}
for all  $k\in\{1,\ldots,p\}$.
\end{proof}


Now, we apply the above results to determine the new sufficient conditions for
functions to be in the class of $p$-valent starike  of order $\alpha$.
\begin{theorem}\label{t4}
Let $f(z)=z^p+\sum_{n=p+1}^\infty a_nz^n$ be analytic in $\mathbb
D$. Assume that
\begin{equation}\label{1t4}
    \left|\arg\{f^{(p)}(z)\}\right|<\frac{\pi\alpha_0}{2}
    , \quad z\in\mathbb D.
\end{equation}
for some $\alpha_0\in(0,3/2]$, and assume that the sequence
$\{\alpha_k\}_{k=1}^{k=p}$ is defined by
\begin{equation*}
        \alpha_k+\frac{2}{\pi}\tan^{-1}\frac{\alpha_k}{k}=\alpha_{k-1}.
\end{equation*}
Then we have
\begin{equation}\label{2t4}
    \left|\arg\left\{\frac{zf^{(p-s+1)}(z)}{f^{(p-s)}(z)}\right\}\right|
    <\frac{\pi}{2}(\alpha_s+\alpha_{s-1})
     \quad z\in\mathbb D,
\end{equation}
for all $s\in \{2,\ldots,p\}$. Furthermore, if there exists a
positive integer $\sigma\in \{2,\ldots,p\}$ such that
$\alpha_{\sigma}+\alpha_{\sigma-1}\leq1$, then
\begin{equation}\label{3t4}
    \left|\arg\left\{\frac{zf'(z)}{f(z)}\right\}\right|
    <\frac{\pi(\alpha_{p-1}+\alpha_p)}{2}\leq\frac{\pi}{2}, \quad z\in\mathbb D,
\end{equation}
or $f(z)$ is $p$-valently strongly starlike  of order
$\alpha_{p-1}+\alpha_p$.
\end{theorem}

\begin{proof}

Applying \eqref{111t3}  for $k=s$ and for $k=s-1$, $s\in
\{2,\ldots,p\}$, we obtain
\begin{equation*}
    \left|\arg\left\{\frac{f^{(p-s+1)}(z)}{z^{s-1}}\right\}\right|<
     \frac{\alpha_{s-1}\pi}{2},
     \quad \left|\arg\left\{\frac{f^{(p-s)}(z)}{z^{s}}\right\}\right|<
     \frac{\alpha_{s}\pi}{2}
     , \quad z\in\mathbb D.
\end{equation*}
It is known that the number sequence  $\{\alpha_k\}_{k=1}^{k=p}$ is
a decreasing sequence. If $\alpha_0=3/2$, then $\alpha_1=1$. If
$\alpha_0<3/2$, then $\alpha_1<1$ and $\alpha_2<\alpha_1<1$.
Therefore $\alpha_k+\alpha_{k-1}<2$ for all $k\in \{2,\ldots,p\}$.
Then, we have

\begin{eqnarray*}
  \left|\arg\left\{\frac{zf^{(p-s+1)}(z)}{f^{(p-s)}(z)}\right\}\right|
  &=& \left|\arg\left\{\frac{z^s}{f^{(p-s)}(z)}\frac{f^{(p-s+1)}(z)}{z^{s-1}}\right\}\right| \\
   &\leq& \left|\arg\left\{\frac{f^{(p-s)}(z)}{z^{s}}\right\}\right|+
   \left|\arg\left\{\frac{f^{(p-s+1)}(z)}{z^{s-1}}\right\}\right|
    \\
   &<& \frac{\pi}{2}(\alpha_s+\alpha_{s-1}), \quad z\in\mathbb D
\end{eqnarray*}
for all $s\in \{2,\ldots,p\}$. This gives \eqref{2t4}. The sum
$\alpha_s+\alpha_{s-1}$ decreases when $s$ runs over  $
\{1,\ldots,p\}$.  Therefore, if there exists a positive integer
$\sigma\in \{1,\ldots,p\}$ such that
$\alpha_{\sigma}+\alpha_{\sigma-1}\leq1$, then
\begin{equation*}
    \left|\arg\left\{\frac{zf^{(p-s+1)}(z)}{f^{(p-s)}(z)}\right\}\right|<\frac{\pi}{2}(\alpha_s+\alpha_{s-1})\leq\frac{\pi}{2}
\end{equation*}
for all $s\in \{\sigma,\ldots,p\}$. Applying this for $s=p$ gives
\eqref{3t4}.
\end{proof}
{\bf{Remark 1.}} Theorem \ref{t4} holds also for $s=1$ with
additional assumption $\alpha_0+\alpha_1<2$. Otherwise \eqref{2t4}
has no sense.

{\bf{Remark 2.}} The condition $\alpha_s+\alpha_{s-1}<1$ depends on
$\alpha_0\in(0,3/2]$ and on a positive integer $p$. For example

\begin{equation*}
\begin{array}{lll}
  &\alpha_0=3/2,\quad &\widetilde{\alpha}_0=1,\\
  &\alpha_1= 1,\quad  &\widetilde{\alpha}_1=0.638\ldots,\\
  &\alpha_2= 0.76\ldots,\quad &\widetilde{\alpha}_2= 0.486\ldots,\\
  &\alpha_3= 0.65\ldots,\quad &\widetilde{\alpha}_3=0.401\ldots,\\
  &\alpha_4= 0.56\ldots,\quad &\widetilde{\alpha}_4=0.347\ldots,\\
  &\alpha_5= 0.5\ldots,\quad  &\widetilde{\alpha}_5=0.309\ldots,\\
\end{array}
\end{equation*}
The limit \eqref{limit} implies that a $\sigma$ such that
$\alpha_s+\alpha_{s-1}<1$ for $s\geq \sigma$, exists whenever $p$ is
sufficiently large. Note that some related resuls  conditions for
functions to be in the class of $p$-valent starike where published in \cite{Nunokawa128, Nunokawa170, Nunokawa204}.

\begin{theorem}\label{t5}
Assume that
\begin{equation*}
    f(z)=\sum_{k=1}^\infty a_kz^k, \quad z\in\mathbb D.
\end{equation*}
If $a_{s-1}=0$ and $a_s\neq0$ for some positive integer $s\geq 2$
and
\begin{equation}\label{1t5}
    \left|\arg\{f^{(s)}(z)\}\right|<\frac{\pi}{2}\left(\delta+\frac{2}{\pi}\tan^{-1}\delta\right)
    , \quad z\in\mathbb D.
\end{equation}
for some $\delta>0$ such that
$2\delta+\frac{2}{\pi}\tan^{-1}\delta<2$. Then we have
\begin{equation}\label{2t5}
    \left|\arg\left\{\frac{zf^{(s)}(z)}{f^{(s-1)}(z)}\right\}\right|
    <\frac{\pi}{2}\left(2\delta+\frac{2}{\pi}\tan^{-1}\delta\right),
     \quad z\in\mathbb D.
\end{equation}
\end{theorem}

\begin{proof}

We have
\begin{equation*}
    f^{(n)}(z)=\sum_{k=0}^\infty \frac{(n+k)!}{k!}a_{n+k}z^k, \quad z\in\mathbb D.
\end{equation*}
Therefore,
\begin{equation*}
    f^{(s)}(z)=s!a_s+(s+1)!a_{s+1}z+\cdots, \quad z\in\mathbb D
\end{equation*}
and
\begin{equation*}
    f^{(s-1)}(z)=s!a_sz+\frac{(s+1)!}{2!}a_{s+1}z^2+\cdots, \quad z\in\mathbb
    D.
\end{equation*}
because $a_{s-1}=0$. Write
\begin{equation*}
    p(z)=\frac{f^{(s-1)}(z)}{s!a_s z}, \quad z\in\mathbb D.
\end{equation*}
Then $p(0)=1$ and
\begin{equation*}
    f^{(s)}(z)=s!a_s(p(z)+zp'(z)).
\end{equation*}
From \eqref{1t1}, we have
\begin{equation}\label{3t5}
    \left|\arg\{p(z)+zp'(z)\}\right|<
   \frac{\pi}{2}\left(\delta+\frac{2}{\pi}\tan^{-1}\delta\right)
    , \quad z\in\mathbb D.
\end{equation}
 If there exists a point
$z_0\in\mathbb D$, such that
\begin{eqnarray*}
     |\arg\{p(z)\}|&<&\frac{\delta\pi}{2},\quad |z|<|z_0|,\nonumber\\
    |\arg\{p(z_0)\}|&=&\frac{\delta\pi}{2},
\end{eqnarray*}
then from Lemma \ref{l1}, we have
\begin{equation*}
    \frac{z_0p'(z_0)}{p(z_0)}=\frac{2ik\arg\left\{p(z_0)\right\}}{\pi}
\end{equation*}
for some $k\geq 1$. For the case
$\arg\{p(z_0)\}=\frac{\delta\pi}{2}$, we have
\begin{equation*}
    \tan^{-1}\delta\leq
    \arg\left\{1+\frac{z_0p'(z_0)}{p(z_0)}\right\}=\arg\left\{1+i\delta
    k\right\}<\pi/2
\end{equation*}
because $k\geq 1$. Therefore, we have
\begin{equation}\label{xxxx}
     \arg\left\{
    p(z_0)\left(1+\frac{z_0p'(z_0)}{p(z_0)}\right)\right\}\\
    =\arg\left\{
    p(z_0)\right\}+ \arg\left\{1+\frac{z_0p'(z_0)}{p(z_0)}\right\}.
\end{equation}

This gives
\begin{eqnarray*}
  \arg\left\{
    p(z_0)+z_0p'(z_0)\right\}
    &=& \arg\left\{
    p(z_0)\left(1+\frac{z_0p'(z_0)}{p(z_0)}\right)\right\}\\
    &=&\arg\left\{
    p(z_0)\right\}+ \arg\left\{1+\frac{z_0p'(z_0)}{p(z_0)}\right\}\\
    &=&\frac{\delta\pi}{2}+\arg\left\{1+i\delta k\right\} \\
   &\geq&\frac{\delta\pi}{2}+\tan^{-1}\delta\\
   &=&\frac{\pi}{2}\left(\delta+\frac{2}{\pi}\tan^{-1}\delta\right).
\end{eqnarray*}
This contradicts \eqref{3t5}. For the case
$\arg\{p(z_0)\}=-\frac{\delta\pi}{2}$, applying the same method as
the above we can obtain
\begin{equation*}
    \arg\left\{
    p(z_0)+z_0p'(z_0)\right\}\leq
    -\frac{\pi}{2}\left(\delta+\frac{2}{\pi}\tan^{-1}\delta\right)
\end{equation*}
This contradicts \eqref{3t5} too. Therefore, we have
\begin{equation*}
     \left|\arg\left\{p(z)\right\}\right|
     =\left|\arg\frac{f^{(s-1)}(z)}{s!a_s z}\right|
     =\left|\arg\frac{f^{(s-1)}(z)}{z}\right|
     <\frac{\delta\pi}{2}
    , \quad z\in\mathbb D.
\end{equation*}
From this we have
\begin{eqnarray*}
  \left|\arg\left\{\frac{zf^{(s)}(z)}{f^{(s-1)}(z)}\right\}\right| &=&
  \left|\arg\left\{f^{(s)}(z)\right\}\right|
  +\left|\arg\left\{\frac{f^{(s-1)}(z)}{z}\right\}\right|\leq \\
   &<& \frac{\pi}{2}\left(\delta+\frac{2}{\pi}\tan^{-1}\delta\right)+\frac{\delta\pi}{2} \\
   &=& \frac{\pi}{2}\left(2\delta+\frac{2}{\pi}\tan^{-1}\delta\right)
\end{eqnarray*}
which ends the proof.
\end{proof}

{\bf{Remark 3.}} In Theorem \ref{t5} we consider $\delta$ such that
 $\delta>0$ and $2\delta+\frac{2}{\pi}\tan^{-1}\delta<2$. These conditions
give $0<\delta<0.787\ldots$. From this, the fast factor in
\eqref{1t5} satisfies
\begin{equation*}
    0<\delta+\frac{2}{\pi}\tan^{-1}\delta<1.21\ldots.
\end{equation*}
For example, if we take $\delta=\sqrt{3}/3$ and $s=2$,
$f(z)=a_2z^2+a_3z^3+\cdots$, then Theorem \ref{t5} becomes
\begin{equation*}
    \left|\arg\{f''(z)\}\right|<\frac{\pi(1+\sqrt{3})}{6}
    \quad\Rightarrow\quad
     \left|\arg\frac{zf''(z)}{f'(z)}\right|
     =\left|\arg\frac{2a_2+6a_3z+\cdots}{2a_2+3a_3z+\cdots}\right|
    <\frac{\pi(2+\sqrt{3})}{6}
    , \quad z\in\mathbb D.
\end{equation*}
\section{Ethical Approval}
Not applicable.

\section{Competing interests}
The authors declare they have no financial interests. The authors have no conflicts of interest to declare that are relevant to the content of this article.

\section{Authors' contributions}
The   authors declare they have equal contribution in the paper.

\section{Funding}
Not applicable.

\section{Availability of data and materials}
Data sharing is not applicable to this article as no datasets were generated or analysed during the current study.


\end{document}